\documentclass{amsart}
\usepackage{amsmath, amsfonts, amssymb, amsthm}
\usepackage{hyperref}
\newtheorem{theorem}{Theorem}[section]
\newtheorem{lemma}[theorem]{Lemma}
\newtheorem{proposition}[theorem]{Proposition}
\newtheorem{corollary}[theorem]{Corollary}
\newtheorem{example}[theorem]{Example}

\newtheorem{definition}[theorem]{Definition}

\title{Images of polynomials with involution on $2\times 2$ matrices}

\author[L. Centrone]{Lucio Centrone}
	\address{Dipartimento di Matematica, Universit\`a degli Studi di Bari, via Orabona 4, 70125, Bari, Italy}\email{lucio.centrone@uniba.it}
	
\author[T. C. de Mello]{Thiago Castilho de Mello}
	\address{Instituto de Ci\^encia e Tecnologia, Universidade Federal de S\~ao Paulo, Av. Cesare M. Giulio Lattes, 1201, 12247014,	S\~ao Jos\'e dos Campos, SP, Brazil}\email{tcmello@unifesp.br}



\date{\today}

\begin{document}

\begin{abstract}
    Let $\mathbb{F}$ be a field and let $M_2(\mathbb{F})$ be the algebra of $2\times 2$ matrices endowed with an involution of the first kind. We study the image of multilinear $*$-polynomials evaluated on $M_2(\mathbb{F})$. For the transpose involution over $\mathbb{R}$, we show that the image is either a proper vector subspace or contains a basis of $M_2(\mathbb{R})$.
    For the symplectic involution over quadratically closed fields or over $\mathbb{R}$, we prove that the image is always a vector space, namely one of $\{0\}$, $\mathbb{F}$, $sl_2(\mathbb{F})$ or $M_2(\mathbb{F})$. As a byproduct, we complete a theorem of Bre\v{s}ar and Klep describing the linear span of the image of a $*$-polynomial on finite dimensional central simple algebras with involution of the first kind. Their result excluded algebras of dimensions 4 and 16; we settle both cases, extending the description to all dimensions greater than 1 (over $\mathbb{R}$ for the transpose involution, and over quadratically closed fields or $\mathbb{R}$ for the symplectic involution). We also classify all Lie skew-ideals of $M_4(\mathbb{F})$ over fields of characteristic zero.
\end{abstract}

\maketitle

\section{Introduction}

Bre\v{s}ar and Klep \cite{BresarKlep_nullstellensatz} studied the linear span of the set of values of a noncommutative polynomial evaluated on matrix algebras. They characterize the set of noncommutative polynomials having each possible image. This paper was followed by the paper \cite{K-BMR} and somehow revived an old conjecture of L'vov\cite{Dniester} asking whether the image of a multilinear polynomial evaluated on matrices is a vector space, similar to a problem  attributed to Kaplansky, leading to what is now known as \textit{L'vov-Kaplansky conjecture}.

This conjecture asks whether the image of a multilinear polynomial evaluated on the $n\times n$ matrix algebra is always a vector space, which is equivalent to say that such image is one of the following: $\{0\}$, $\mathbb{F}$, $sl_n(\mathbb{F})$ or $M_n(\mathbb{F})$.

Although simple to state, such conjecture has only a few known cases of positive solutions, namely the case of $2\times 2$ matrices over quadratically closed fields or over the real numbers \cite{K-BMR, MalevM} and for arbitrary $n\times n$ matrices over $\mathbb C$ when the polynomial has degree up to 3 \cite{Shoda, AlbertMuckenhoupt, DykemaKlep, Vitas3}.
Similar problems for nonassociative algebras were also considered. Here we mention, as examples, the image of multilinear polynomials evaluated on the algebra of octonions \cite{MalevO} and the image of polynomials on low dimensional Lie algebras \cite{CentroneFindik_Images}.

Given the difficulty in proving such conjecture, alternative approaches to this kind of problem have been considered. For instance, analogous problems have been considered for other associative algebras with complete results being obtained, for instance, for the algebra of upper triangular matrices \cite{GargatedeMello} and for the algebra of quaternions \cite{MalevQ}. Also other structures such as gradings were considered to handle this problem, aiming both the original L'vov-Kaplansky conjecture \cite{GargateMello2} and variations of the problem to the setting of graded polynomials evaluated on graded algebras \cite{CentronedeMello_imagegraded, FagundesKoshlukov_imagegraded}. 

Another approach to study the image of a polynomial is to consider the minimal number $n$ such that any element of the linear span of the image of $f$ is a sum of at most $n$ elements in the image of $f$. More precisely, we say that $f$ has \textit{finite width} in an $\mathbb{F}$-algebra $\mathcal{A}$ if there exists an $n\geq 1$ such that every element in $\mathrm{span}_{\mathbb{F}}f(\mathcal{A})$ is a linear combination of $n$ elements from $f(\mathcal{A})$. The smallest such $n$ is called the \textit{width} of $f$ in $\mathcal{A}$. Determining such width is known as \textit{Waring problem} for associative algebras (see for instance \cite{BresarMartinez_survey} for a recent survey on the topic).
With the above terminology, the image of a given polynomial $f$ evaluated on an algebra $\mathcal{A}$ is a vector space if it has width 1. If every polynomial $f$ has finite width in $\mathcal{A}$, then we define the \emph{Waring constant} of $\mathcal{A}$ to be the supremum of the widths of all polynomials.
In particular, the L'vov-Kaplansky conjecture is true for a given algebra $\mathcal{A}$ if the Waring constant of $\mathcal{A}$ is 1. 
Such notions are inspired by analogous problems in group theory, where one studies the image of word maps on groups. For instance, in a noteworthy paper, Larsen, Shalev and Tiep~\cite{LarsenShalevTiep_Waring} proved that for any non-trivial word $w$, there exists a constant $N = N(w)$ such that every element of every finite simple group of order at least $N$ is a product of $N$ values of $w$, thus establishing the Waring problem for finite simple groups.

Although in the paper of Bre\v sar and Klep \cite{BresarKlep_nullstellensatz} the authors deal with $*$-polynomials evaluated on algebras with involution, they do not consider the problem of describing the image of such polynomial map, but only the linear span of such image.

In the present paper, we consider the problem of describing the image of a multilinear $*$-polynomial evaluated on $2\times 2$ matrix algebras.

The linear span of the image of a $*$-polynomial is a Lie-skew ideal (that is invariant under commutators by skew-symmetric elements). In the same paper \cite{BresarKlep_nullstellensatz}, the authors describe the Lie-skew ideals of matrix algebras of order different from 2 and 4. 

In this article we extend the above mentioned result for matrices of order 2 and 4. For the particular case of $2\times 2$ matrices over $\mathbb{R}$ with the transpose involution, we prove that the image of such polynomial is a proper vector subspace or contains a basis of $M_2(\mathbb{R})$. For $M_2(\mathbb{F})$, with $\mathbb{F}$ quadratically closed or $\mathbb{F} = \mathbb{R}$, endowed with the symplectic involution, we prove that the image is one of the following: $\{0\}$, $\mathbb{F}$, $sl_2(\mathbb{F})$ or $M_2(\mathbb{F})$.

\section{Preliminaries}

Let $\mathbb{F}$ be a field, and $\mathcal{A}$ be an $\mathbb{F}$-algebra. An \textit{involution} (of the first kind) in $\mathcal{A}$ is an $\mathbb{F}$-linear map $*:\mathcal{A} \longrightarrow \mathcal{A}$ satisfying $(a^*)^*=a$ and $(ab)^* = b^*a^*$, for any $a, b\in \mathcal{A}$. If $\mathcal{A}$ is endowed with an involution $*$ we say that $\mathcal{A}$ is a \textit{$*$-algebra} (or an \textit{algebra with involution}). If $\mathbb{F}$ has characteristic not 2, the algebra $\mathcal{A}$ can be decomposed as a direct sum $\mathcal{A} = \mathcal{A}^+\oplus \mathcal{A}^-$, of symmetric and skew-symmetric elements of $\mathcal{A}$.

Recall that an involution on $M_n(\mathbb{F})$ is equivalent either to the symplectic or transpose involution \cite[Corollary 3.1.58]{Rowen_PI}. The symplectic involution occurs only for even $n$  and it is not equivalent to the transpose involution if the characteristic of the field is not 2. The symplectic involution is defined on $M_{2k}(\mathbb{F})$ by
\[\begin{pmatrix}
    A & B \\ C & D
\end{pmatrix}^s = \begin{pmatrix}
    D^t & -B^t\\-C^t & A^t
\end{pmatrix},\] where $M^t$ denotes the usual transpose of a matrix $M\in M_k(\mathbb{F})$. In particular, for $k=1$, if  $A = \begin{pmatrix}
    a & b \\ c & d
\end{pmatrix}\in M_2(\mathbb{F})$, we have $A^s = \begin{pmatrix}
    a & b \\ c & d
\end{pmatrix}^s = \begin{pmatrix}
    d & -b \\ -c & a
\end{pmatrix}$ and if such matrix $A$ is invertible, we have $A^{s} = \det(A) A^{-1}$.

If $X$ is a set, we denote by $\mathbb{F}\langle X \rangle$ the free associative algebra, freely generated by $X$ over $\mathbb{F}$. This is the algebra of noncommutative polynomials in the variables of $X$. To our purposes (that is, to define polynomial maps on a $*$-algebra) we need to endow the algebra of polynomials with a suitable structure of $*$-algebra. To do this, we will consider the algebra $\mathbb{F}\langle X|*\rangle$, freely generated by $X\cup X^*$ over $\mathbb{F}$. Here $X^* = \{x^*\,|\, x\in X\}$. Now $\mathbb{F}\langle X|*\rangle$ is an algebra with involution. Notice that if we define $y_i = \dfrac{x_i+x_i^*}{2}$ and $z_i = \dfrac{x_i-x_i^*}{2}$, then $\mathbb{F}\langle X|*\rangle$ is isomorphic to the algebra $\mathbb{F}\langle Y;Z\rangle $ freely generated by $Y\cup Z$, where $Y=\{y_1,y_2, \dots\}$ and $Z = \{z_1,z_2, \dots\}$ over $\mathbb{F}$ are the sets of symmetric and skew-symmetric variables (that is, the involution $*$ satisfies $y_i^* = y_i$ and $z_i^* = -z_i$, respectively). This algebra will be called the \textit{free algebra with involution}. Now we can consider polynomials $f\in \mathbb{F}\langle Y;Z\rangle$ in symmetric and skew-symmetric variables.

If $\mathcal{A} = \mathcal{A}^+\oplus \mathcal{A}^-$ is a $*$-algebra, given a $*$-polynomial $f(y_1, \dots, y_k, z_1, \dots, z_l)\in \mathbb{F}\langle Y;Z\rangle$, it defines in a natural way a map by evaluation
\[\begin{array}{cccc}
    f: & (\mathcal{A}^+)^k\times(\mathcal{A}^-)^l & \longrightarrow & \mathcal{A} \\
     & (a_1, \dots, a_k, b_1, \dots, b_l) & \mapsto &f(a_1, \dots, a_k, b_1, \dots, b_l)
\end{array}\]

Here we are interested in studying the image of such map, which we will denote by $f(\mathcal{A})$, especially when the $*$-polynomial is multilinear (that is, if in each monomial of $f$, every variable occurs exactly once), we want to know if its image is a vector space. In other words, we are looking for a $*$-version of the L'vov-Kaplansky conjecture.

Notice that the above set has some noteworthy properties. For instance, the following is a trivial but important fact.

\begin{lemma}\label{*-invariant}
    Let $f\in \mathbb{F}\langle Y;Z\rangle$ be a $*$-polynomial. Then the image of $f$ is invariant under $*$-automorphisms of $\mathcal{A}$, that is automorphisms of algebras $\varphi:\mathcal{A} \longrightarrow \mathcal{A}$ such that $\varphi(\mathcal{A}^+) = \mathcal{A}^+$ and $\varphi(\mathcal{A}^-) = \mathcal{A}^-$.
\end{lemma}

Now if we assume the algebra $\mathcal{A}$ to be central and simple, we know by Skolem-Noether theorem that any such automorphism is given by a conjugation by an invertible element $u$. The next lemma provides conditions on such element $u$ so that the automorphism it defines is a $*$-automorphism.

\begin{lemma}\label{centralsym}
    Let $\mathcal{A}$ be a finite dimensional central simple $\mathbb{F}$-algebra endowed with an involution $*$. Then an invertible element $u\in \mathcal{A}$ satisfies $u^{-1}\mathcal{A}^+u\subseteq \mathcal{A}^{+}$ and $u^{-1}\mathcal{A}^-u\subseteq \mathcal{A}^{-}$ if and only if $uu^* \in \mathbb{F}$.
\end{lemma}

\begin{proof}
    Since conjugation by $u$ preserves symmetric matrices, if $a$ is symmetric, then $u^{-1}au = (u^{-1}au)^* = u^{*}a^*(u^{-1})^* = u^{*}a(u^{-1})^*$. As a consequence, $uu^*a=auu^*$ for any symmetric matrix $a$. Similarly, if $b$ is skew-symmetric, then $u^{-1}bu = -(u^{-1}bu)^* = -u^{*}b^*(u^{-1})^* = u^{*}b(u^{-1})^*$ and again $uu^*b=buu^*$ for any skew-symmetric matrix $b$. As a consequence, we obtain that $uu^*$ is a central element, that is $uu^*\in \mathbb{F}$. Conversely if $u$ is an invertible element satisfying $uu^*\in \mathbb{F}$, it follows directly that conjugation by $u$ preserves symmetric and skew-symmetric elements of $\mathcal{A}$.
\end{proof}

Notice that in the above lemma if $\mathcal A=M_n(\mathbb{F})$, it is not necessary to consider both conditions (on $\mathcal A^+$ and on $\mathcal A^-$), but only one of them, depending on whether the involution is transpose or symplectic.

\begin{proposition}
    Let $\mathcal A = M_n(\mathbb{F})$ endowed with an involution $*$ and let $u\in \mathcal A$ be an invertible element. Then
    \begin{enumerate}
        \item If $*$ is the transpose involution and $u$ satisfies $u^{-1}\mathcal A^+u\subseteq \mathcal A^+$, then $u^{-1}\mathcal A^-u\subseteq \mathcal A^-$. 
        
        \item If $*$ is the symplectic involution, and $u$ satisfies $u^{-1}\mathcal A^-u\subseteq \mathcal A^-$, then $u^{-1}\mathcal A^+u\subseteq \mathcal A^+$.
    \end{enumerate}
\end{proposition}

\begin{proof}
    It follows from the fact that if $u$ satisfies (1) (respectively (2)) then $u^*u$ commutes with all transpose-symmetric (respectively symplectic-skew-symmetric) elements. As a consequence, $u^*u$ must be a scalar element and the result follows.
\end{proof}
In particular, Lemma \ref{centralsym} implies that the image of a $*$-polynomial $f$ is invariant under conjugation by elements $u$ such that $u^{-1} = u^*$. These elements are called unitary.
For instance, if $\mathcal{A}$ is $M_n(\mathbb{F})$ with the transpose involution, it means that the image of $f$ on $M_n(\mathbb{F})$ is invariant under the orthogonal group $O(n,\mathbb{F})$, while if $\mathcal{A}$ is $M_{2n}(\mathbb{F})$ endowed with the symplectic involution, then the image of $f$ is invariant under the action of the symplectic group $Sp(2n, \mathbb{F})$.

Another important fact, as proved in \cite[Theorem 2.5]{BresarKlep_nullstellensatz}, is that the linear span of the image of a $*$-polynomial $f$ is a Lie skew-ideal. Such a notion is the appropriate extension of Lie ideals to the setting of algebras with involution.

\begin{definition}
    Let $\mathcal{A}$ be an algebra with involution. A subspace $\mathcal{L}$ of $\mathcal{A}$ is a Lie skew-ideal if for any $x\in \mathcal{L}$ and any skew-symmetric element $a\in \mathcal{A}^-$, $[x,a]\in \mathcal{L}$.
\end{definition}

A straightforward consequence of \cite[Theorem 2.5]{BresarKlep_nullstellensatz} is the following.
\begin{theorem}
    Let $f(y_1, \dots, y_k,z_1, \dots, z_l)\in \mathbb{F}\langle Y;Z\rangle$ be a $*$-polynomial and let $\mathcal{A}$ be an algebra with involution. Then, the linear span of the image of $f$ on $\mathcal{A}$ is a Lie skew-ideal.
\end{theorem}

Hence, to determine the possible subspaces that arise as the linear span of the image of a given polynomial, we need to consider Lie skew-ideals that are invariant under conjugation by unitary elements. Depending on the underlying field, it suffices to verify one of the above conditions, as shown in the following result.

\begin{proposition}(\cite[Proposition 3.6.]{BresarKlep_nullstellensatz})
    Let $\mathcal A$ be a real or complex Banach algebra with  an $\mathbb{R}$-linear involution $*$. If a closed linear subspace $\mathcal L$ of $\mathcal A$ is closed under conjugation with unitaries, then $\mathcal L$ is a Lie skew-ideal of $\mathcal A$.
\end{proposition}

Recall that for algebras with involution, the set of $*$-polynomial identities and central polynomials of $2\times 2$ matrices were previously described in \cite{ColomboKoshlukov_involutions, Brandao}.

\section{On commutators of symmetric and skew-symmetric matrices with the transpose involution}

Let now $\mathcal{A}$ be an $\mathbb{F}$-algebra and $a,b\in \mathcal{A}$. The commutator of $ a$ and $b$ is defined by  $[a,b]:=ab-ba$. If $U,V\subseteq \mathcal{A}$, we define the set $[U,V]$ to be the linear span of the set $\{[u,v], u\in U, v\in V\}$.

Almost a century ago, Shoda \cite{Shoda} showed that over fields of characteristic zero, every element of $[M_n(\mathbb{F}), M_n(\mathbb{F})]$ can be written as a single commutator (in other words, sums of commutators are still commutators). Such result was generalized to fields of arbitrary characteristic by Albert and Muckenhoupt in \cite{AlbertMuckenhoupt}, and recently for other classes of algebras, such as block-triangular matrices \cite{FagundesMello}. Nowadays a more direct proof of the above fact is known \cite{Trace0}.

In this section we consider the algebra $\mathcal{A} = M_n(\mathbb{R})$ endowed with the transpose involution. We show that a result similar to that of Shoda, Albert and Muckenhoupt holds in this setting.

\begin{proposition}\label{comm}
    Let $A\in M_n(\mathbb{R})$ be a trace zero symmetric matrix. Then there exists a symmetric matrix $B$ and a skew-symmetric matrix $C$ such that $A = [B,C]$.
\end{proposition}

\begin{proof}
    The proof follows the same ideas as \cite{Trace0}, with suitable adaptations. Since the above reference was not formally published (but simply made available on the webpage of the author), we decided to present here the complete proof. 

    The proof is done by induction on $n$. If $A=0$ or $n=1$, the result is obvious.

    Assume $n>1$ and that the result true for matrices of order $n-1$.

    Since $\mathrm{tr}(A) = 0$ and $A\neq 0$, $A$ is a nonscalar symmetric matrix. In particular, $A$ has at least two distinct eigenvalues, say $\lambda_1>0$ and $\lambda_2 < 0$ (since the eigenvalues are real and sum to zero). By the intermediate value theorem applied to the Rayleigh quotient $v\mapsto v^TAv/\|v\|^2$, there exists a unit vector $v$ such that $v^TAv = 0$. Extending $v$ to an orthonormal basis, we obtain an orthogonal matrix $P$ such that $P^{-1}AP$ is a symmetric trace zero matrix with $(1,1)$-entry zero, that is
    \[P^{-1}AP = \begin{pmatrix}
        0 & u^T \\ u & A_0
    \end{pmatrix}.\] 

    In particular, $A_0$ is a symmetric trace zero matrix of order $n-1$. By induction hypothesis, $A_0 = [B_0, C_0]$, where $B_0$ is a symmetric matrix and $C_0$ is a skew symmetric matrix. Without loss of generality, we may assume $B_0$ to be invertible, by adding a suitable scalar matrix if necessary.

    Now, \[P^{-1}AP = \left[\begin{pmatrix}
        0 & 0^T \\ 0 & B_0
    \end{pmatrix}, \begin{pmatrix}
        0 & -u^TB_0^{-1} \\ B_0^{-1}u & C_0
    \end{pmatrix}\right].\]

    The first matrix above is symmetric and the second is skew-symmetric.

    So  $P^{-1}AP = [B_1, C_1]$, where $B_1$ and $C_1$ are symmetric and skew-symmetric, respectively. Hence, $A = [PB_1P^{-1}, PC_1P^{-1}] = [B,C]$, and $B$ and $C$ are symmetric and skew-symmetric, respectively, since $P$ preserves symmetric and skew-symmetric matrices.

\end{proof}

A similar result holds for skew-symmetric matrices.

\begin{proposition}\label{commskew}
    Let $A\in M_n(\mathbb{R})$ be a skew-symmetric matrix. Then there exist symmetric matrices $B$ and $C$ such that $A = [B,C]$.
\end{proposition}

\begin{proof}
    It is well known that every real skew-symmetric matrix is congruent, via an orthogonal matrix, to a block-diagonal canonical form. More precisely, there exists an orthogonal matrix $P\in O(n,\mathbb{R})$ such that
    \[
    P^{t} A P = \begin{pmatrix} 0 & \lambda_1 \\ -\lambda_1 & 0 \end{pmatrix} \oplus \cdots \oplus \begin{pmatrix} 0 & \lambda_k \\ -\lambda_k & 0 \end{pmatrix} \oplus 0_{n-2k},
    \]
    where $\lambda_1, \dots, \lambda_k \in \mathbb{R}$.
    Now, define the following symmetric matrices:
    \[
    D = \dfrac{1}{2}\left[\begin{pmatrix} \lambda_1 & 0 \\ 0 & -\lambda_1 \end{pmatrix} \oplus \cdots \oplus \begin{pmatrix} \lambda_k & 0 \\ 0 & -\lambda_k \end{pmatrix} \oplus 0_{n-2k}\right]
    \]
    and
    \[
    S = \begin{pmatrix} 0 & 1 \\ 1 & 0 \end{pmatrix} \oplus \cdots \oplus \begin{pmatrix} 0 & 1 \\ 1 & 0 \end{pmatrix} \oplus 0_{n-2k}.
    \]
    A direct computation in each $2\times 2$ block shows that
    \[
    \left[\begin{pmatrix} \lambda_i/2 & 0 \\ 0 & -\lambda_i/2 \end{pmatrix}, \begin{pmatrix} 0 & 1 \\ 1 & 0 \end{pmatrix}\right] = \begin{pmatrix} 0 & \lambda_i \\ -\lambda_i & 0 \end{pmatrix},
    \]
    and the commutator on the zero blocks vanishes. Therefore $[D, S] = P^{t} A P$.
    Setting $B = PDP^{t}$ and $C = PSP^{t}$, we have that $B$ and $C$ are symmetric, since $D$ and $S$ are symmetric and $P$ is orthogonal. Moreover,
    \[
    [B, C] = [PDP^{t}, PSP^{t}] = P[D, S]P^{t} = P(P^{t}AP)P^{t} = A.  \qedhere
    \]
\end{proof}

As a consequence of the above results, we present a key reduction that will be used throughout the paper. Since every skew-symmetric matrix in $M_n(\mathbb{R})$ (with respect to the transpose involution) can be written as a commutator of two symmetric matrices by Proposition \ref{commskew}, we can reduce the study of the image of a $*$-polynomial to that of an ordinary polynomial evaluated on symmetric matrices only.

\begin{theorem}\label{reduction}
Let $f(y_1, \dots, y_k, z_1, \dots, z_l)\in \mathbb{R}\langle Y;Z\rangle$ be a multilinear $*$-polynomial. Then the image of $f$ on $M_n(\mathbb{R})$ (endowed with the transpose involution) equals the image of the multilinear polynomial in $k+2l$ symmetric variables \[g(y_1, \dots, y_{k+2l}) = f(y_1, \dots, y_k, [y_{k+1},y_{k+2}], \dots, [y_{k+2l-1},y_{k+2l}])\] evaluated on symmetric matrices.
\end{theorem}

\begin{proof}
    The inclusion $\mathrm{Im}(g) \subseteq \mathrm{Im}(f)$ is clear, since the commutator of two symmetric matrices is skew-symmetric. Conversely, given any element $f(a_1, \dots, a_k, b_1, \dots, b_l)\in \mathrm{Im}(f)$, where $a_i$ are symmetric and $b_j$ are skew-symmetric, by Proposition \ref{commskew} we can write each $b_j = [B_j, C_j]$ for some symmetric matrices $B_j$ and $C_j$. Then $f(a_1, \dots, a_k, b_1, \dots, b_l) = g(a_1, \dots, a_k, B_1, C_1, \dots, B_l, C_l) \in \mathrm{Im}(g)$.
\end{proof}

\section{Lie skew-ideals of $M_n(\mathbb{F})$}

In the paper \cite{BresarKlep_nullstellensatz} the authors prove the following fact.

\begin{theorem}\cite[Theorem 3.13]{BresarKlep_nullstellensatz}\label{Lieskew}
Let $\mathcal{A}$ be a finite dimensional central simple algebra with involution of the first kind, and let $\mathcal{L}$ be a Lie skew-ideal of $\mathcal{A}$. If $\dim_{\mathbb{F}} \mathcal{A}\neq 4, 16$ and $\mathrm{char}(\mathbb{F}) \neq 2, 3$ then $\mathcal{L}$ is either $0, \mathcal{Z}, \mathcal{K}, [\mathcal{S},\mathcal{K}], \mathcal{S}, \mathcal{Z} + \mathcal{K}, [\mathcal{A},\mathcal{A}]$ or $\mathcal{A}$.
\end{theorem}

Here $\mathcal S, \mathcal K$ and $\mathcal Z$ denote the set of symmetric, skew-symmetric and central elements, respectively.

We will now describe the Lie skew-ideals when $\mathcal A = M_4(\mathbb{F})$ endowed with the transpose involution and with $\mathbb{F}$  a field of zero characteristic, which does not satisfy the conditions of the above theorem since it has dimension 16.

Notice that if we consider $\mathcal A=M_n(\mathbb{F})$ endowed with the transpose involution, then a vector subspace $V \subseteq M_n(\mathbb{F})$ is a Lie skew-ideal if and only if $V$ is a $\mathfrak{so}(n)$-submodule of $M_n(\mathbb{F})$.

We are allowed now to use the classical results of representation of Lie algebras when $\mathrm{char} \  \mathbb{F} = 0$. To do this we first observe that  the Lie algebra $\mathfrak{so}(n)$ of skew-symmetric matrices with the product defined by the commutator is a semisimple Lie algebra if $n> 2$. Actually, it is simple if $n=3$ or $n> 4$ and $\mathfrak{so}(4)$ is semisimple, but not simple. 

Now we use the fact that the representation of a semisimple Lie algebra is completely reducible. See for instance \cite[Theorem 9.19]{FultonHarris}.

Now, straightforward computations show that $M_4(\mathbb{F})$ decomposes as a direct sum of irreducible $\mathfrak{so}(4)$-modules
\[M_4(\mathbb{F}) = \mathbb{F}\cdot I \oplus \mathcal{K}_1\oplus \mathcal{K}_2\oplus [\mathcal{S},\mathcal{K}].\]

Here $\mathcal{K}_1$ and $\mathcal{K}_2$ satisfy $\mathcal{K} = \mathcal{K}_1\oplus \mathcal{K}_2$ and are given by
\[\mathcal{K}_1 = \left\{\begin{pmatrix}
    0 & a & b & c \\
    -a & 0 & c & -b\\
    -b & -c & 0 & a\\
    -c & b & -a & 0
\end{pmatrix}\, ,\, a,b,c\in \mathbb{F}\right\}\] 
\[\mathcal{K}_2 = \left\{\begin{pmatrix}
    0 & a & b & c \\
    -a & 0 & -c & b\\
    -b & c & 0 & -a\\
    -c & -b & a & 0
\end{pmatrix}\, ,\, a,b,c\in \mathbb{F}\right\}\]

As a consequence, we obtain that each $\mathfrak{so}(4)$-submodule of $M_4(\mathbb{F})$ is a direct sum of some of the four components above. In particular, the analog of the above theorem  for $M_4(\mathbb{F})$ is as follows. 

\begin{proposition}\label{Lieskew4}
    Let $\mathbb{F}$ be a field of characteristic zero and let $\mathcal{L}$ be a Lie skew-ideal of $M_4(\mathbb{F})$. Then $\mathcal{L}$ is one of the following: \[0, \mathcal{Z}, \mathcal{K}_1, \mathcal{K}_2, [\mathcal{S},\mathcal{K}], \mathcal{K}, \mathcal{S}, \mathcal{Z} + \mathcal{K}_1, \mathcal{Z} + \mathcal{K}_2, \mathcal{Z} + \mathcal{K},\] 
    \[\mathcal{K}_1 + [\mathcal{S},\mathcal{K}], \mathcal{K}_2 + [\mathcal{S},\mathcal{K}], \mathcal{Z} + \mathcal{K}_1 + [\mathcal{S},\mathcal{K}], \mathcal{Z} + \mathcal{K}_2 + [\mathcal{S},\mathcal{K}], [\mathcal{A},\mathcal{A}] \text{ or } \mathcal{A}.\]
\end{proposition}

The case $n=2$ is quite different, since $\mathfrak{so}(2)$ is one dimensional (and hence abelian).

As a consequence of Theorem \ref{Lieskew},  in \cite{BresarKlep_nullstellensatz} the authors prove the following result concerning the linear span of the images of $*$-polynomials.

\begin{theorem}\label{biggone}\cite[Theorem 4.10]{BresarKlep_nullstellensatz}\label{BresarKlep_preimages}
Let $\mathcal{A}$ be a finite dimensional central simple algebra with involution of the first kind, let $f \in \mathbb{F}\langle X|*\rangle$ and let us write $\mathcal{L} = \mathrm{span} f(\mathcal{A})$. If $\dim_{\mathbb{F}} \mathcal{A} \neq  1, 4, 16$ and $char(F) = 0$, then exactly one of the following eight possibilities holds:
\begin{enumerate}
    \item $f\in Id(\mathcal{A})$; in this case, $\mathcal{L}=0$;
    \item $f\in \mathrm{Cen}(A)$; in this case, $\mathcal{L}=\mathcal{Z}$;
    \item $f\in \mathrm{Skew}\mathbb{F} \langle X|*\rangle + Id(A)$ and $f \not \in Id(A)$; in this case $\mathcal{L} = \mathcal{K}$;
    \item  $f\in \mathrm{Skew}\mathbb{F} \langle X|*\rangle + Cen(A)$ and $f \not \in Cen(A)$; in this case $\mathcal{L} = \mathcal{Z}+ \mathcal{K}$;
    \item $f\in \mathrm{Sym}\mathbb{F} \langle X|*\rangle + Id(A)$, $f \not \in Id(A)$ and $f$ is cyclically equivalent to an element of $Id(\mathcal{A})$; in this case $\mathcal{L} = [\mathcal{S}, \mathcal{K}]$;
    \item $f\in \mathrm{Sym}\mathbb{F} \langle X|*\rangle + Id(A)$, $f \not \in Cen(A)$ and $f$ is not cyclically equivalent to an element of $Id(\mathcal{A})$; in this case $\mathcal{L} = \mathcal{S}$;
    \item $f \not \in \mathrm{Sym}\mathbb{F} \langle X|*\rangle + Id(A)$, $f \not \in \mathrm{Skew} \mathbb{F} \langle X|*\rangle + Id(A)$, and $f+f^*$ is cyclically equivalent to an element of $Id(A)$; in this case $\mathcal{L} = [\mathcal{A}, \mathcal{A}]$;
    \item  $f \not \in  \mathrm{Sym} \mathbb{F} \langle X|*\rangle + Id(A)$, $f \not \in \mathrm{Skew} \mathbb{F} \langle X|*\rangle + Id(A)$, $f\not \in \mathrm{Skew} \mathbb{F} \langle X|*\rangle + Cen(A)$ and $f+f^*$ is not cyclically equivalent to an element of $Id(\mathcal{A})$; in this case $\mathcal{L} = \mathcal{A}$.    
\end{enumerate}    
\end{theorem}

Recall that two polynomials $f$ and $g$ in $\mathbb{F}\langle X|*\rangle$ are cyclically equivalent if $f -g $ is a sum of commutators in $\mathbb{F}\langle X|*\rangle$ \cite[Definition 1.2]{KlepSchweighofer}.

We claim that the above theorem holds, more generally, for any central simple algebra $\mathcal{A}$ with $\dim_{\mathbb{F}} \mathcal{A} > 1$, equipped with an involution $*$ of the first kind, subject to the following conditions: if $*$ is symplectic, we require  and that $\mathbb{F}$ be either quadratically closed or $\mathbb{R}$; if $*$ is the transpose involution, we require $\mathbb{F} = \mathbb{R}$.

Indeed, for the symplectic involution, \cite[Lemma 3.12]{BresarKlep_nullstellensatz} establishes that the Lie skew-ideals of $M_{2n}(\mathbb{F})$ are precisely the eight sets appearing in Theorem~\ref{Lieskew}. Consequently, when $*$ is symplectic, or when $*$ is the transpose involution and $\dim_{\mathbb{F}} \mathcal{A} \neq 4, 16$, the proof of the above theorem carries over verbatim from the original argument.

It remains to treat the case $\dim_{\mathbb{F}} \mathcal{A} = 16$ with the transpose involution, that is, $\mathcal{A} = M_4(\mathbb{F})$. The only modification needed is the following observation: the linear span of the image of a polynomial evaluated on $M_4(\mathbb{F})$ is invariant under conjugation by elements of $O(4)$, and so the argument reduces to the result below.

\begin{lemma}
    Let $\mathbb{F}$ be a field of characteristic zero, and let  $\mathcal{L}$ be a Lie skew-ideal of $M_4(\mathbb{F})$. If $\mathcal{L}$ is closed under conjugation by elements of $O(4)$, then $\mathcal{L}$ is one of the following:
    \[0, \mathcal{Z}, \mathcal{K}, [\mathcal{S}, \mathcal{K}], \mathcal{S}, \mathcal{Z}+\mathcal{K}, [\mathcal{A}, \mathcal{A}], \mathcal{A}\]
\end{lemma}
\begin{proof}
    One just needs to observe that $\mathcal{K}_1$ and $\mathcal{K}_2$ are not invariant under $O(4)$. In fact, $P=\begin{pmatrix}
        0 & 0 & 1 & 0\\
        0 & 0 & 0 & -1\\
        1 & 0 & 0 & 0 \\
        0 & 1 & 0 & 0
    \end{pmatrix}$ is an element of $O(4)$ and $P^{-1}\mathcal{K}_1 P = \mathcal{K}_2$. In particular, if an element of $\mathcal{K}_1$ is in $\mathcal{L}$ then $\mathcal{K}$ is in $\mathcal{L}$.
\end{proof}

\begin{corollary}
   Theorem \ref{BresarKlep_preimages} also holds if we assume $\dim_{\mathbb{F}} \mathcal A = 16$. 
\end{corollary}

The remaining case $\dim_{\mathbb{F}}\mathcal{A} = 4$ will follow as a consequence of the main results of this paper (see Corollary \ref{BK_dim4}).

\section{Invariant cones of $M_2(\mathbb R)$}

As mentioned in the introduction, we study the image of a $*$-polynomial $f$, assuming it is multilinear. For multilinear polynomials, the image has the property of being invariant under scalar multiplication (actually, it just needs to be linear in one of its variables).

These two features together will motivate our next definition (which is an adaptation from \cite[Definition 1]{K-BMR}).

\begin{definition}
    Let $\mathcal{A}$ be an algebra endowed with an involution $*$. Assume that a group $G$ acts linearly on $\mathcal{A}$ and that $\mathcal{A}^+$ and $\mathcal{A}^-$ are invariant under the action of $G$. We say that a subset $C\subseteq A$ is 
    \begin{enumerate}
        \item a \emph{cone} if $C$ is closed under nonzero scalar multiplication.
        \item a \emph{$G$-invariant cone} if $C$ is a cone that is closed under the action of $G$.
        \item an \emph{irreducible $G$-invariant cone} if $C$ is a $G$-invariant cone that does not properly contain any $G$-invariant cone. 
    \end{enumerate}
\end{definition}

    We will omit the $G$ in the above definition, if it is clear from the context.

    If $\mathcal{A}  = M_n(\mathbb{R})$, we are interested in describing the irreducible invariant cones for $G=O(n, \mathbb{R})$ (the orthogonal group) when it is endowed with the transpose involution and for $G=Sp(2k,\mathbb{R})$ (the symplectic group, for $n=2k$) when it is endowed with the symplectic involution.

Notice that if a $G$-invariant cone intersects non-trivially an irreducible invariant cone $C$, then it must contain $C$. Also, the set of irreducible invariant cones consists of a partition of $\mathcal{A}$ and we can represent each irreducible $G$-invariant cone by any of its elements. In this case, we say that the cone is generated by such an element.

\begin{example}
    Let $\mathcal{A}  = M_n(\mathbb{R})$ be endowed with the transpose involution. Then the following subsets are irreducible $O(n,\mathbb{R})$-invariant cones.
    \begin{enumerate}
        \item $\{0\}$ (the trivial one).
        \item $\mathbb{R}\setminus \{0\}$ - the set of nonzero scalar matrices.
        \item The set of $n\times n$ traceless symmetric matrices.
    \end{enumerate}
\end{example}

Let us denote $O(n,\mathbb{R})$ and  $SO(n,\mathbb{R})$ simply by $O(n)$ and $SO(n)$, respectively.
In this section, we will determine all irreducible $O(2)$-invariant cones of $M_2(\mathbb{R})$, considering the $O(2)$-action by conjugation, by describing their generators.

We will consider the linear basis $\mathcal{B}=\{I, E, e_1, e_2\}$ of $M_2(\mathbb R)$ where $e_1 = \begin{pmatrix}
    1 & 0 \\ 0 & -1
\end{pmatrix}$, $e_2 = \begin{pmatrix}
    0 & 1 \\ 1 & 0
\end{pmatrix}$ and $E = e_1e_2$.  Notice that $E$ is skew-symmetric, whereas $e_1$ and $e_2$ are symmetric.

Also, we will denote by $V$ the 2-dimensional real vector space with orthonormal basis $\{e_1, e_2\}$. Note that $V$ equals the space of symmetric traceless matrices in $M_2(\mathbb{R})$.

\begin{lemma}\label{O(2)-orbit}
    Let $A, B\in M_2(\mathbb{R})$ so that
    $A = \alpha_0 I + \alpha_{12}E + u$ and $B = \beta_0 I + \beta_{12}E + v$, with $u,v\in V$. Then $A$ and $B$ are in the same $O(2)$-orbit if and only if $\alpha_0 = \beta_0$, $\alpha_{12} = \pm \beta_{12}$ and $||u|| = ||v||$. 
\end{lemma}

\begin{proof}
    Notice that any element of $O(2)$ is either an element of $SO(2)$ or is a product of an element of $SO(2)$ by a reflection, and without loss of generality, we can assume the reflection to be the matrix $e_1$.
    
    Notice that $O(2)$ acts on the $\mathbb{R}$-vector space $V={\langle e_1, e_2\rangle}$ by conjugation. Indeed, if     
    $P=\begin{pmatrix}
        \cos (\theta) & -\sin(\theta)\\ \sin(\theta) & \cos(\theta)
    \end{pmatrix}\in SO(2)$, then straightforward computations show that
    
    \begin{align*}
        Pe_1P^{-1} = &+\cos(2\theta)e_1+\sin(2\theta)e_2  \\
        Pe_2P^{-1} =& -\sin(2\theta)e_1+\cos(2\theta)e_2
    \end{align*}
    Also, if $P=e_1$, then it acts trivially on $e_1$ and acts by inverting the sign on $e_2$.

    By the above computations, one  also obtains that if $u,v\in V$, then $||u|| = ||v||$ if and only if there exists $P\in O(2)$ such that $v = PuP^{-1}$. 

    Now, we have two cases to consider. First, suppose $P\in SO(2)$ such that $PAP^{-1}=B$, then we obtain $\alpha_0=\beta_0$, $\alpha_{12}=\beta_{12}$ and $PuP^{-1}=v$ yielding $||u||=||v||$. Now, if $P\in O(2)\setminus SO(2)$, we may write it as $P=Qe_1$, where $Q\in SO(2)$. Then, if $PAP^{-1} = B$, we have $Qe_1Ae_1Q^{-1}=B$, we obtain $\alpha_0=\beta_0$, $\alpha_{12}=-\beta_{12}$ and $Qe_1ue_1Q^{-1}=v$ yielding again $||u||=||v||$. It remains now to prove the converse. Suppose first $\alpha_0 = \beta_0$, $\alpha_{12} = \beta_{12}$  and $||u|| = ||v||$, then in order to show $A$ and $B$ are in the same $O(2)$ orbit we just need to find an element $P$ of $SO(2)$ mapping $u$ to $v$. This $P$ can be given as above, for $\theta$ being half of the angle between $u$ and $v$. Hence, for such $P$ we have $P^{-1}AP = B$. Otherwise, if $\beta_{12} = -\alpha_{12}$, then we first apply conjugation by $e_1$, which changes the sign of the coefficient of $E$, but preserves the norm of $u$, then we find an element $P\in SO(2)$ mapping $u$ to $v$.     
\end{proof}

Now we describe the irreducible invariant cones on $M_2(\mathbb R)$. To do this, we note that each irreducible invariant cone can be represented by any of its elements, since any such element cannot lie inside another irreducible invariant cone.

\begin{lemma}\label{irred.cones.M2}
    Each nontrivial irreducible invariant cone of $M_2(\mathbb R)$ is generated by exactly one of the following matrices:
    \begin{enumerate}
        \item[(0)] $0 = \begin{pmatrix}
            0 & 0 \\0 & 0
        \end{pmatrix}$
        \item $I=\begin{pmatrix}
            1 & 0 \\ 0 & 1
        \end{pmatrix}$ (the cone of nonzero scalar matrices)
        \item $E= \begin{pmatrix}
            0 & 1 \\ -1 & 0
        \end{pmatrix}$  (the cone of nonzero skew-symmetric matrices)
        \item $e_1=\begin{pmatrix}
            1 & 0 \\ 0 & -1
        \end{pmatrix}$ (the cone of nonzero symmetric traceless matrices)
        \item $\begin{pmatrix}
            a+1 & 0 \\ 0 & a-1
        \end{pmatrix}, a> 0$ 
        \item $\begin{pmatrix}
            a+s & 1 \\ -1 & a-s
        \end{pmatrix}, a, s \geq 0$ with $a\neq 0$ or $s\neq 0$.
    \end{enumerate}
\end{lemma}

\begin{proof}
    Any matrix $A$ in $M_2(\mathbb R)$ can be uniquely written as $A = \alpha I+ \beta E + v$, for some scalars $\alpha, \beta \in \mathbb{R}$ and some $v\in V$. Recall that any symmetric matrix is conjugated to a diagonal matrix under the action of $O(2)$. Now, since the subspace of skew-symmetric matrices is invariant under $O(2)$, we obtain that $A$ is in the same orbit of $ \begin{pmatrix}
        \alpha+\gamma & \beta\\ -\beta & \alpha-\gamma
    \end{pmatrix}$ for some $\beta, \gamma \in \mathbb{R}$.  Moreover, notice that both diagonal and skew-symmetric matrices are invariant under conjugation by $E$. On the other hand, conjugating $\begin{pmatrix}
        s & 0 \\ 0 & -s
    \end{pmatrix}$ by $E$ we obtain $\begin{pmatrix}
        -s & 0 \\ 0 & s
    \end{pmatrix}$. As a consequence, we may assume $\gamma\geq 0$. Since diagonal matrices commute among them, conjugating $A$ by $e_1=\begin{pmatrix}
        1 & 0 \\ 0 & -1
    \end{pmatrix}$ we have $\begin{pmatrix}
        \alpha+\gamma & -\beta \\ \beta & \alpha-\gamma
    \end{pmatrix}$; as a consequence, we may assume $\beta \geq 0$ too. Hence, up to conjugation by an element of $O(2)$, any matrix $A\in M_2(\mathbb{R})$ can be written as $A = \begin{pmatrix}
        \alpha+\gamma & \beta\\ -\beta & \alpha-\gamma
    \end{pmatrix}$ for some $\alpha\in \mathbb{R}$ and $\beta, \gamma \geq 0$.

    {Let us now show any irreducible invariant cone is one among those listed above.} 

    We first assume $\beta = 0$.
    
    If $\gamma =0$, we obtain $(1)$ -- the cone of scalar matrices if $\alpha \neq 0$ and $(0)$ -- the trivial cone if $\alpha = 0$.

    If $\gamma \neq 0$, we may divide by $\gamma$ and obtain an element of the form $\begin{pmatrix}
        \alpha +1 & 0 \\ 0 &\alpha - 1 
    \end{pmatrix}$. Notice that we may assume $\alpha \geq 0$. Indeed, if $\alpha<0$, after multiplying by $-1$ and conjugating by $e_1$, we obtain $\begin{pmatrix}
        -\alpha +1 & 0 \\ 0 &-\alpha - 1 
    \end{pmatrix}$. As a consequence, we obtain $(3)$ if $\alpha =0$ or $(4)$ if $\alpha \neq 0$. Also, notice that if $\alpha_1\neq \alpha_2$ are two positive scalars, then $\begin{pmatrix}
        \alpha_1 +1 & 0 \\ 0 & \alpha_1-1
    \end{pmatrix}$ and  $\begin{pmatrix}
        \alpha_2 +1 & 0 \\ 0 & \alpha_2-1
    \end{pmatrix}$ generate distinct invariant cones.

    Now we assume $\beta \neq 0$. By dividing by $\beta$ we can write $A = \begin{pmatrix}
        \alpha+\gamma & 1\\ -1 & \alpha-\gamma
    \end{pmatrix}$. Now, if $\gamma <0$, conjugating by $E$, it simply ends up in a change of the sign of $\gamma$, and we may assume $\gamma \geq 0$. Finally, if $\alpha<0$, after multiplying $A$ by $-1$ and  conjugating by $ \begin{pmatrix}
        0 & 1 \\ 1 & 0 
    \end{pmatrix}$ we obtain the matrix $A'=\begin{pmatrix}
        -\alpha+\gamma & 1 \\ -1 & -\alpha-\gamma 
    \end{pmatrix}$. As a consequence, we may also assume $\alpha \geq 0$. 

    If both $\alpha = \gamma = 0$, we obtain $(2)$, otherwise, we obtain $(5)$.

    To finish the proof, we just need to observe that if $(\alpha_1,\gamma_1)$ and $(\alpha_2, \gamma_2)$ are two distinct pairs of non-negative real parameters, the cone generated by $\begin{pmatrix}
        \alpha_1+\gamma_1 & 1 \\ -1 & \alpha_1-\gamma_1
    \end{pmatrix}$ does not contain  $\begin{pmatrix}
        \alpha_2+\gamma_2 & 1 \\ -1 & \alpha_2-\gamma_2 \end{pmatrix}$. Indeed, assuming they are in the same irreducible invariant cone, since scalar matrices are central, we must have $\alpha_2 = \alpha_1$. Furthermore, {after conjugating and multiplying by a non-zero scalar, the only possible changes preserving $\gamma$ and the non-diagonal entries are the signs of the $\gamma_i$. Since the latter must be positive, we have $\gamma_1 = \gamma_2$.}
\end{proof}

\section{Image of multilinear polynomials on symmetric $2\times 2$ matrices over $\mathbb{R}$ with the transpose involution}

By Theorem \ref{reduction}, the problem of describing the image of a multilinear $*$-polynomial on $M_2(\mathbb{R})$ (with the transpose involution) reduces to describing the image of a multilinear polynomial evaluated on symmetric $2\times 2$ matrices. In this section we address this problem.

We recall Lemma 1.34 of \cite{survey}, which will help us in proving the main result of this section.
    
    \begin{lemma}\label{dim2}
        Let $V_i$ (for $1\leq i\leq m$) and $V$ be linear spaces over an arbitrary field $K$. Let $f: \prod\limits_{i=1}^m V_i\rightarrow V$ be a multilinear map. Assume there exist two points in $Im(f)$ which are not proportional. Then $Im(f)$ contains a $2$-dimensional plane. In particular, if $V$ is $2$-dimensional, then $Im(f)=V$.
    \end{lemma}

We will also need the following technical lemma, which is based on ideas given in the proof of \cite[Theorem 1]{MalevQ} and in the proof of \cite[Theorem 5]{MalevJ}. 

\begin{lemma}\label{nonzero}
    Let $U$ and $V$ be $\mathbb{R}$-vector spaces and $f:V^m\longrightarrow \mathbb{R} \oplus U$ be an $m$-linear map. Assume that there exist $x_1, \dots, x_m, y_1, \dots, y_m\in V$ such that $f(x_1, \dots, x_m)\in \mathbb{R}^*$ and $f(y_1, \dots, y_m)\not \in \mathbb{R}$. Then, there exists an index $i$ and vectors $r_1, \dots, r_m, r_i^*\in V$ such that 
    \[
    f(r_1, \dots, r_m) = 1 \in \mathbb{R}^* \quad \text{and} \quad f(r_1, \dots, r_{i-1}, r_i^*, r_{i+1}, \dots, r_m)\not \in \mathbb{R}.
    \]
\end{lemma}

\begin{proof}

    We define the following maps:
    \begin{align*}
        f_0 &= f(x_1, \dots, x_m), \\
        f_1(z_1) &= f(z_1, x_2, \dots, x_m), \\
        &\vdots \\
        f_i(z_1, \dots, z_i) &= f(z_1, \dots, z_i, x_{i+1}, \dots, x_m), \\
        &\vdots \\
        f_m(z_1, \dots, z_m) &= f(z_1, \dots, z_m).
    \end{align*}
    
    Note that $\operatorname{Im}(f_j)\subseteq \operatorname{Im}(f_{j+1})$ for all $j$ and $\operatorname{Im}(f_m) = \operatorname{Im}(f)$.

    Since $\operatorname{Im}(f) \not\subseteq \mathbb{R}$, there exists a minimal index $i$ such that $\operatorname{Im}(f_{i-1})\subseteq \mathbb{R}$ and $\operatorname{Im}(f_{i})\not\subseteq \mathbb{R}$.
    
    If $i=1$, since $f_0\neq 0$, we are done.

    Assume now $i>1$. Let $r_1, \dots, r_i\in V$ such that 
    \[
    f(r_1, \dots, r_i,x_{i+1}, \dots, x_m) = a + u \not\in \mathbb{R} \quad (\text{i.e., } a\in\mathbb{R},\ u\neq 0).
    \]
    
    If $f(r_1, \dots, r_{i-1}, x_i, \dots, x_m)\neq 0$, we are done. So assume 
    \[
    f(r_1, \dots, r_{i-1}, x_i, \dots, x_m)= 0.
    \]

    Replace $r_{i-1}$ by $r_{i-1}' = r_{i-1}+x_{i-1}$. By linearity:
    \begin{align*}
        a_1 &= f(r_1, \dots, r_{i-2}, r'_{i-1},x_{i},\dots, x_m) \\
            &= f(r_1, \dots,r_{i-2}, x_{i-1},\dots, x_m) + f(r_1, \dots, r_{i-1},x_i,\dots, x_m) \\
            &= f(r_1, \dots,r_{i-2}, x_{i-1},\dots, x_m)
    \end{align*}
    and
    \begin{align*}
        a_2 &= f(r_1, \dots, r_{i-2}, r'_{i-1},r_{i}, x_{i+1},  \dots, x_m) \\
            &= f(r_1, \dots,r_{i-2}, x_{i-1},r_i, x_{i+1}\dots, x_m) + f(r_1, \dots, r_{i-1},r_i, x_{i+1}\dots, x_m).
    \end{align*}
    
    Notice that in $a_2$, the second summand is non-real. If the first summand is real, then $a_2$ is non-real.

    Notice now, $a_2$ and $a_1$ are evaluations of $f$ on a same $m$-tuple but the position $i$. If $a_1$ is non-zero, we are done. Otherwise, we replace $r_{i-2}$ by $r_{i-2}+x_{i-2}$ and repeat the same process above.
    
    This procedure eventually stops since $f(x_1, \dots, x_m)\neq 0$ and we have $f(r_1, \dots, r_m) = r\in \mathbb{R}^*$. To get $f(r_1, \dots, r_m) = 1$, we only need to replace $r_1$ by $r_1/r$ and the proof is complete.
\end{proof}

Let $Y = \{y_1, y_2, \dots\}$. We will now compute the image of a multilinear polynomial $f(y_1, \dots, y_m)$ evaluated on symmetric $2\times 2$ matrices over $\mathbb{R}$.

We will also need the following simple, but interesting result:

\begin{proposition}\label{twosymmetric}
    Let $\mathbb{F}$ be an arbitrary field. Then any matrix $A\in M_n(\mathbb{F})$ can be decomposed as a product of two symmetric matrices.
\end{proposition}

\begin{proof}
    By \cite[Theorem 1]{TausskyZassenhaus}, for any $A\in M_n(\mathbb{F})$, there exists a non-singular invertible symmetric matrix $S\in M_n(F)$ such that $A^T = SAS^{-1}$. We claim that $SA$ is also symmetric. Indeed, since $S$ is symmetric $S^T=S$. Also, the above equation implies $SA = A^TS = A^TS^T = (SA)^T$, proving the claim. To finish the proof, one just needs to observe that $A = S^{-1}SA$, which is a product of two symmetric matrices. 
\end{proof}

Before stating and proving the main result of this section, let us present some examples. The examples below are consequences of Lemma \ref{dim2}, of the Albert-Muckenhoupt-Shoda Theorem, and of Proposition \ref{twosymmetric}.

\begin{example}\label{example1}
    \begin{enumerate}

        \item It is well-known that the standard polynomial $s_4(y_1, y_2,y_3,y_4)$ in four variables is a polynomial identity for $M_2(\mathbb
        {R})$.
        
        \item The image of the polynomial $[y_1,y_2]$ is the set of skew-symmetric matrices.
        
        \item The image of the polynomial $f(y_1,y_2,y_3,y_4) = [y_1,y_2][y_3,y_4]$ is $\mathbb{R}$.

        \item The image of the polynomial $f(y_1,y_2, y_3)  =  [y_1, y_2, y_3]$ is the set of all symmetric trace zero matrices.

        \item The image of the polynomial $f(y_1, \dots, y_6) = [y_1,y_2,y_3][y_4, y_5, y_6]$ is the set $\mathbb{R} \oplus \mathbb{R}\cdot E$.

        \item The image of the polynomial $f(y_1, y_2, y_3, y_4) = [y_1y_2, y_3y_4]$ is $sl_2(\mathbb{R})$,  the set of all trace zero matrices.

        \item The image of the polynomial $f(y) = y$ is the set of all symmetric matrices.

        \item The image of the polynomial $f(y_1, y_2) = y_1y_2$ is the set $M_2(\mathbb {R})$.
        
    \end{enumerate}
\end{example}

Our main result of this section is the following.

\begin{theorem}\label{main_transpose}
    Let $f(y_1, \dots, y_m)\in \mathbb{R}\langle Y \rangle$ be a multilinear polynomial. Then the image of $f$ evaluated on symmetric $2\times 2$ matrices over $\mathbb{R}$ is one of the following:
    \begin{enumerate}
        \item $\{0\}$;
        \item The set of all skew-symmetric matrices $\langle E \rangle$;
        
        \item $\mathbb{R}$, viewed as the set of scalar matrices;
        \item The set $\langle e_1,  e_2 \rangle$ of trace zero symmetric matrices;
        \item $\mathbb{R}\oplus \mathbb{R}\cdot E$;
        
        \item The set $sl_2(\mathbb{R})$ of all trace zero matrices;
        \item The set of all symmetric matrices;
        \item A set containing a basis of $M_2(\mathbb{R})$.
    \end{enumerate}
\end{theorem}

\begin{proof}
    In Example \ref{example1} we have already shown that each of the above cases can be realized as images of polynomials evaluated on symmetric matrices. We will show that there are no other possibilities.
    
    Assume $f$ is not an identity. If the image of $f$ is contained in a 1-dimensional vector space, then the image of $f$ is the vector space itself. Since the only 1-dimensional vector spaces that are closed under conjugation by $O(2)$ are $\mathbb{R}$ and $\mathbb{R}\cdot E$, these are the only possibilities of 1-dimensional images.
    
    So we may assume that there are two  linearly independent elements in the image of $f$. If the image of $f$ is contained in the subspace generated by these two elements, then the image of $f$ is exactly this subspace, by Lemma \ref{dim2}. 

    Now observe that the elements of the basis $\{I, e_1, e_2\}$ of the space of symmetric matrices either commute or anticommute, and the product of any two of them is a scalar multiple of an element of $\mathcal{B} = \{I, E, e_1, e_2\}$. As a consequence, the evaluation of any monomial of $f$ on elements of $\{I, e_1, e_2\}$ is a scalar multiple of an element of $\mathcal{B}$, and moreover any two monomials of $f$ evaluated on the same tuple of basic elements either agree or differ by a sign. It follows that each evaluation of $f$ on elements of $\{I, e_1, e_2\}$ is a scalar multiple of an element of $\mathcal{B}$. Hence, if $f$ has two linearly independent elements in its image, there exist two elements of $\mathcal{B}$ in its image. We analyze all possibilities.

    Since $e_1$ and $e_2$ are in the same $O(2)$-orbit, we have that  $e_1$ is in the image of $f$, if and only if $e_2$ is in the image of $f$. So one possible 2-dimensional subspace in the image is $\langle e_1, e_2\rangle$ - the set of trace zero symmetric matrices. As a consequence, the only other possibility for a 2-dimensional space is the space $\langle 1, e_1e_2\rangle = \mathbb{R} \oplus\mathbb{R} \cdot E $.

    Now we may assume that the image of $f$ is contained in a 3-dimensional subspace and that there are three linearly independent elements in the image of $f$. Again, by the above observation and by the fact that $e_1$ and $e_2$ are in the same $O(2)$-orbit, there are only 2 cases to consider: either the image of $f$ is contained in the subspace spanned by  $\{1, e_1, e_2\}$ or in the subspace spanned by $\{e_1, e_2, E\}$.

    In the first case, we have $1, e_1, e_2$ in the image of $f$. We will show that the image of $f$ is the set of symmetric matrices. First, notice that $\langle e_1, e_2 \rangle$ is contained in the image of $f$. Indeed, given $u = \alpha_1e_1+\alpha_2e_2\in \langle e_1, e_2\rangle$, by Lemma \ref{O(2)-orbit}, we obtain $u$ is conjugated to $\sqrt{\alpha_1^2+\alpha_2^2}e_1$ and the latter is in the image of $f$ because it is a cone. By Lemma \ref{nonzero} there exist $r_1, \dots, r_m$ and $r_i^*$, symmetric matrices, such that $f(r_1, \dots, r_m)\in \mathbb{R}^*$ and $f(r_1, \dots, r_{i-1}, r_i^*, r_{i+1}, \dots, r_m) = a+ u$ for some non-zero $u\in \langle e_1, e_2\rangle$. Now  we take $\tilde r_i = r_i^*-a\cdot r_i$ and we obtain that $f(r_1, \dots, r_{i-1}, \tilde r_i, r_{i+1}, \dots, r_m) = u\neq 0$. Hence, for any $\alpha, b\in \mathbb{R}$, $b+\alpha u = f(r_1, \dots, r_{i-1}, br_i +\alpha\tilde r_i, r_{i+1}, \dots, r_m)$. So given $b+v\in \mathbb{R} \oplus \langle e_1, e_2\rangle$, let $\alpha\in \mathbb{R}$ such that $||v|| = ||\alpha u||$, Then we have by Lemma \ref{O(2)-orbit} $b+ v$ and $b+\alpha u$ are in the same $O(2)$-orbit. Since the last is in the image of $f$, the same holds for the former. So the image of $f$ is $\mathbb{R}\oplus \langle e_1, e_2\rangle$.

    Now we assume that $e_1, e_2$ and $E$ are in the image of $f$. Notice that a similar version of Lemma \ref{nonzero} holds in which we guarantee the  existence of an index $i$ and symmetric matrices $r_1, \dots, r_m, r_i^*$  such that $f(r_1, \dots, r_m) = E$ and $f(r_1, \dots, r_{i-1}, r_i^*,r_{i+1}, \dots, r_m) = u \in \langle  e_1, e_2\rangle\setminus \{0\}$. Again, using the fact that $E$ is invariant under conjugation by $SO(2)$, and that the image of $f$ is an invariant $O(2)$-cone, by a similar argument as above, we obtain that any element in $\mathbb R\cdot E \oplus \langle e_1, e_2\rangle$ is in the image of $f$.

    If none of the above cases occurs, the image of $f$ contains a basis of $M_2(\mathbb{R})$. 
    \end{proof}

\section{Image of multilinear polynomials on $2\times 2$ matrices with the symplectic involution}

The case of images of $*$-polynomials evaluated on $M_2(\mathbb{F})$ endowed with the symplectic involution is a straightforward consequence of the description of images of (ordinary) multilinear polynomials evaluated on $M_2(\mathbb{F})$, as obtained in \cite{K-BMR} for quadratically closed fields and extended to the real numbers $\mathbb R$ in \cite{MalevM}.  Also, in \cite{MalevM} the author proves that for an arbitrary field $\mathbb F$ the image is either $\{0\}$, $\mathbb{F}$ or contains $sl_2(\mathbb{F})$.

\begin{theorem}
    Let $\mathbb{F}$ be a quadratically closed field or $\mathbb{F} = \mathbb{R}$. If $f \in \mathbb{F}\langle Y;Z\rangle$ is a multilinear $*$-polynomial evaluated on $M_2(\mathbb{F})$ endowed with the symplectic involution, then the image of $f$ is either $\{0\}$, $\mathbb{F}$, $sl_2(\mathbb{F})$ or $M_2(\mathbb{F})$.
\end{theorem}

\begin{proof}   
Let us first observe that, for $M_2(\mathbb{F})$ endowed with the symplectic involution, the symmetric elements are precisely the scalar matrices, whereas the skew-symmetric elements are exactly $sl_2(\mathbb{F})$. Since every symmetric element is scalar, if $f(y_1, \dots, y_m, z_1, \dots, z_n)\in \mathbb{F}\langle Y;Z\rangle$ is a multilinear $*$-polynomial, its image coincides with the image of the polynomial $g(z_1, \dots, z_n) = f(1, \dots, 1, z_1,\dots, z_n)$, where $1$ denotes the $2\times 2$ identity matrix. By the Albert--Muckenhoupt--Shoda theorem, every traceless matrix is a commutator. Hence, the image of $g$ on traceless matrices is precisely the image of the polynomial 
\[h(x_1, \dots, x_{2n}) = g([x_1,x_2], \dots, [x_{2n-1},x_{2n}])\] evaluated on  $M_2(\mathbb{F})$, which is known by \cite{K-BMR, MalevM} to be one among $\{0\}$, $\mathbb{F}$, $sl_2(\mathbb{F})$ or $M_2(\mathbb{F})$. Now the proof follows.
\end{proof}

We can now complete the extension of the Bre\v{s}ar-Klep theorem (Theorem \ref{BresarKlep_preimages}) to the remaining case $\dim_{\mathbb{F}} \mathcal{A} = 4$.

\begin{corollary}\label{BK_dim4}
    Let $\mathcal{A} = M_2(\mathbb{F})$ be endowed with an involution of the first kind and let $f$ be a multilinear $*$-polynomial. Then the linear span of $f(\mathcal{A})$ is one of the eight subspaces listed in Theorem \ref{BresarKlep_preimages}, provided one of the following holds:
    \begin{enumerate}
        \item $*$ is the transpose involution and $\mathbb{F} = \mathbb{R}$;
        \item $*$ is the symplectic involution and $\mathbb{F}$ is an arbitrary field.
    \end{enumerate}
\end{corollary}

\begin{proof}
    If $\mathcal{A}$ is endowed with the transpose involution, Theorem \ref{main_transpose} shows that the image of any multilinear $*$-polynomial evaluated on $M_2(\mathbb{R})$ is either one of the subspaces listed in (1)--(7), or it contains a basis of $M_2(\mathbb{R})$. In each case, the linear span of the image is one among $0, \mathcal{Z}, \mathcal{K}, [\mathcal{S}, \mathcal{K}], \mathcal{Z} + \mathcal{K}, [\mathcal{A}, \mathcal{A}], \mathcal{S}$ or $\mathcal{A}$, which are exactly the eight possibilities appearing in Theorem \ref{BresarKlep_preimages}. If $\mathcal{A}$ is endowed with the symplectic involution, the conclusion follows from the theorem above, since the only possible images are $\{0\}, \mathbb{F}, sl_2(\mathbb{F}), M_2(\mathbb{F})$ which correspond respectively to $0, \mathcal{Z}, [\mathcal{A}, \mathcal{A}], \mathcal{A}$. These are among the eight subspaces listed in Theorem \ref{BresarKlep_preimages}.
\end{proof}

\section{Final remarks}

It is still an open problem to determine if there is a subset of $M_2(\mathbb{R})$ which is not a subspace and can be realized as the image of a multilinear polynomial.

The next example shows that the analogue of L'vov-Kaplansky conjecture does not hold in general for multilinear $*$-polynomials on matrices. The example is quite simple: the image of $f(z_1,z_2) = z_1z_2$ evaluated on $M_3(\mathbb{F})$.
    
\begin{example}
    Notice that there are $*$-polynomials of degree 2 whose image is not a vector space. Indeed, consider the polynomial $f(z_1,z_2) = z_1z_2$, evaluated on (skew-symmetric elements of) $M_3(\mathbb{F})$, with the transpose involution. In this case, the linear span of the image of $f$ is the whole $M_3(\mathbb{F})$. On the other hand, the image of $f$ lies in the subset \[S = \{M=(m_{ij})\in M_3(\mathbb{F})\,|\, m_{12}m_{23}m_{31}=m_{21}m_{32}m_{13}\},\] which is not the whole $M_3(\mathbb{F})$.
\end{example}

Although the L'vov-Kaplansky conjecture is not true in this case, it would be interesting to compute the width of a $*$-polynomial evaluated on matrices, that is to study the Waring problem in this setting of $*$-polynomials.

\section{Acknowledgements}
This work was carried out while the second-named author was visiting
the University of Bari. The author gratefully acknowledges the support and
the stimulating academic atmosphere provided by the University of Bari. Both authors are grateful to professor  F. Y. Yasumura who suggested the arguments in the proof of Lemma \ref{nonzero}.

\section{Funding}
Both authors were partially supported by CNPq -- Brazil (grant \#405779/2023-2). The second author was also partially supported by Fapesp -- Brazil (grants \#2024/14914-9 and \#2025/02457-5).

\end{document}